\documentclass[amssymb,amsfonts,refcheck,12pt,verbatim,righttag]{amsart}

\usepackage{color}
\usepackage{amssymb}
\usepackage{graphics}
\usepackage{tikz}
\usepackage{multicol}
\usepackage{subfigure}
\usepackage{bm}
\usepackage{algorithm}
\usepackage{algpseudocode}

 \numberwithin{equation}{section} 
\newcommand{\beq}{\begin{equation}}
\newcommand{\eeq}{\end{equation}}

\newcommand{\ben}{\begin{eqnarray}}
\newcommand{\een}{\end{eqnarray}}

\newcommand{\bet}{\begin{eqnarray*}}
\newcommand{\eet}{\end{eqnarray*}}

\DeclareMathOperator{\ord}{ord}
\newcommand{\rad}{\operatorname{rad}}

\newtheorem{thm}{Theorem}[section]
\newtheorem{lem}[thm]{Lemma}

\newtheorem{de}[thm]{Definition}
\newtheorem{rem}[thm]{Remark}

\newtheorem{ques}[thm]{Question}

\usepackage[colorlinks,citecolor=red,hypertexnames=false]{hyperref}

\newcommand{\R}{\mathbb{R}}

\newcommand{\N}{\mathbb{N}}

\DeclareMathOperator{\lcm}{lcm}

\theoremstyle{plain}

\begin{document}
\baselineskip 16pt

\title{On factorial reciprocals in Cantor sets}

\author{Kehao Lin}
\address[]{School of Mathematics and Statistics\\Central South University\\ Changsha, 410083, PR China}
\email{1225357759@qq.com}

\author{Yufeng Wu*}
\address[]{School of Mathematics and Statistics\\HNP-LAMA\\Central South University\\ Changsha, 410083, PR China}
\email{yufengwu.wu@gmail.com}
\thanks{$^*$Corresponding author}

\author{Siyu Yang}
\address[]{School of Mathematics and Statistics\\Central South University\\ Changsha, 410083, PR China}
\email{1040863693@qq.com}


\begin{abstract}
Let $C$ be the middle-third Cantor set. We show that 
\[\left\{\frac{1}{n!}: n\in\N\right\}\cap C=\left\{1, \frac{1}{5!}\right\}.\]
This answers a question recently posed by Jiang \textit{et al.} [J. Lond. Math. Soc., 113(1): e70408, 2026]. Our approach extends to general missing-digit sets, showing that, in any such set, there are only finitely many elements of the form $\frac{1}{n!}$, all of which can be effectively determined. 
\end{abstract}

\maketitle

\section{Introduction}

Given a sequence $S=\{x_n\}_{n=1}^{\infty}$ and a Cantor set $K$ in $\R$, it is generally difficult to describe the intersection $S\cap K$.  This problem is of independent interest and is also related to the famous (and still open) Erd\H{o}s similarity problem. For instance, it is still unknown whether every measurable set in $\R$ with positive Lebesgue measure contains an affine copy of $\left\{\frac{1}{2^n}:n\in\N\right\}$ or  $\left\{\frac{1}{n!}: n\in\N\right\}$. For more details on the Erd\H{o}s similarity problem, we refer the reader to the recent survey \cite{JLM25F}.

Recently, Jiang, Kong, Li, and Wang \cite{JKLW25O} studied the intersections of a class of planar Cantor sets with the unit circle, as well as the intersections of missing-digit sets on $\R$ with certain sequences. Among other results,    they characterized  $ \left\{\frac{1}{n^2}: n\in\N\right\}\cap K$ for  missing-digit sets $K$ satisfying certain assumptions. They also posed the following question.

\begin{ques}\label{Qest}\cite[Question 5.3]{JKLW25O} Let $C$ be the middle-third Cantor set. 
	Determine the intersection
	\[
	\left\{ \frac{1}{n!} : n \in \mathbb{N} \right\} \cap C.
	\]
	Note that $1$ and $\frac{1}{5!}$ are in this intersection.
\end{ques}

We answer Question~\ref{Qest} by proving the following. 

\begin{thm}\label{thmCn!}
	We have 
	\[\left\{\frac{1}{n!}: n\in\N\right\}\cap C=\left\{1, \frac{1}{5!}\right\}.\]
\end{thm}

Our approach naturally generalizes to  all missing-digit sets. 

\begin{de}\label{DefMissing}
	Let $m\in \N_{\geq 3}$ and $D\subset \{0,1,\ldots, m-1\}$ with cardinality  $1<\#D<m$. The missing-digit set with base $m$ and digit set $D$, denoted by $K_{m,D}$, is defined as 
	\[K_{m,D}=\left\{\sum_{i=1}^{\infty}\frac{a_i}{m^i}: a_i\in D, i\geq 1\right\}.\] 
\end{de}

Clearly, the middle-third Cantor set $C$ corresponds to $K_{3, \{0,2\}}$. Theorem~\ref{thmCn!} can be extended to the following more general result. 

\begin{thm}\label{thmmissing}
	Let $m\in \N_{\geq 3}$ and $D\subset \{0,1,\ldots,m-1\}$ with $1<\#D<m$. Then the intersection 
\begin{equation}\label{eq1n1KmdF}
\left\{\frac{1}{n!}: n\in\N\right\}\cap K_{m,D}
\end{equation}
	is  finite and can be effectively determined.
\end{thm}

\begin{rem}\label{reman1}
According to our method, the conclusion of Theorem~\ref{thmmissing} still holds if we replace $\left\{\frac{1}{n!}: n\in\N\right\}$ by $\left\{\frac{a}{n!}: a,n\in\N, \gcd(a,n!)=1\right\}$. 
\end{rem}

After the submission of the present manuscript, Kominers
\cite{Kominers2026}, building on the fixed-prime valuation argument
used in our proof of Theorem~\ref{thmmissing}, abstracted the key denominator obstruction and
established a general criterion for reciprocal sequences in
missing-digit sets. His criterion recovers our factorial result as a
special case and applies, among others, to superfactorials, products
of polynomial values, products of Fibonacci numbers, and products of
$m^k-1$.

In the next section, we first prove~Theorem~\ref{thmmissing}. Then we specialize the argument to obtain Theorem~\ref{thmCn!}. We also provide an algorithm to determine the intersection in Theorem~\ref{thmmissing}. 

\section{Proofs of the main results}

We first introduce some notation. Throughout, $\N$ denotes the set of positive integers $\{1,2,\ldots\}$.  
For $a,b\in \N$ with $a>1$, the {\em $a$-adic valuation} of $b$, denoted by $\nu_a(b)$, is the largest non-negative
integer $s$ such that $a^s\mid b$. Given $a,q\in \N$ with $\gcd(a,q)=1$, we define the {\em multiplicative order of $a$ modulo $q$}, denoted by $\ord_q(a)$, to be the smallest positive integer $t$ such that 
\[a^t\equiv 1\pmod{q}.\] 
For an integer $q\geq 2$, let $\rad(q)$ denote the {\em radical} of $q$, that is, the product of all prime divisors (without multiplicity) of $q$:
\begin{equation}\label{eqradq}
    \rad(q)=\prod_{\substack{p\text{ prime}\\ p\mid q}}p.
\end{equation}
Throughout, to ease notation, we reserve the letter $p$ for prime numbers. Thus, \eqref{eqradq} simplifies to $\rad(q)=\prod_{p\mid q}p$.

Given a reduced fraction $\frac{r}{q}\in (0,1)$ and $m\in\N_{\geq 2}$ with $\gcd(q,m)=1$, it is well known that the $m$-ary expansion of $\frac{r}{q}$ is purely periodic with period length $\ord_q(m)$; see, for example, \cite[Proposition 4.10]{Schleischitz21}. For $i\in \{0,1,\ldots,m-1\}$, let $N_i(r/q)$ denote the number of occurrences of $i$ in a full period of the $m$-ary expansion of $\frac{r}{q}$. Our proofs are based on the following result of Korobov \cite{Korobov70}, which gives an estimate on $N_i(r/q)$.

\begin{lem}\label{LemKorobov}\cite[Theorem 4]{Korobov70}
Let $m\geq 2$ and let $r<q$ be positive integers with $\gcd(mr,q)=1$. Then, for every $i\in\{0,1,\ldots,m-1\}$,
\begin{equation}\label{eqNirq}
\left|N_i\left(\frac rq\right)-\frac{\ord_q(m)}{m}\right|
\leq 2\ord_{\rad(q)}(m).
\end{equation}
\end{lem}

Our application of Korobov's result was inspired by the work of Shparlinski \cite{Shparlinski21} on rational numbers in missing-digit sets. More precisely, in a recent
study on Diophantine approximation on fractal sets, Schleischitz \cite{Schleischitz21} proved that if $\gcd(q,m)=1$, then the missing-digit set $K_{m,D}$ contains at most finitely many rational numbers of the form $\frac{a}{q^k}$, where $a,k\in\N$. By applying Korobov's result (i.e.,
Lemma~\ref{LemKorobov}), Shparlinski \cite{Shparlinski21} managed to prove a quantitative version of Schleischitz's theorem. Our resolution of Question~\ref{Qest} was inspired by Shparlinski's paper. For related  finiteness results, we refer the reader to \cite{LLW23R,JKLW24R,KWL26R}.

Our starting point is the following direct consequence of  Lemma~\ref{LemKorobov}.

\begin{lem}\label{Cor}
Let $m\in \N_{\geq 3}$ and $D\subset\{0,1,\ldots,m-1\}$ with $1<\#D<m$. Let $r<q$ be positive integers with $\gcd(mr,q)=1$. If $\frac{r}{q}\in K_{m,D}$, then
\[
\ord_q(m)\leq 2m\cdot\ord_{\rad(q)}(m).
\]
\end{lem}

\begin{proof}
By the definition of  $K_{m,D}$ (cf. Definition~\ref{DefMissing}), there exists $i\in \{0,1,\ldots, m-1\}\setminus D$. Such a digit $i$ does not appear in the $m$-ary expansion of  $\frac{r}{q}$, implying $N_i(\frac{r}{q})=0$. Then the lemma follows from Lemma~\ref{LemKorobov}.
\end{proof}

For $m\geq3$ and $n\in\N$, define
\begin{equation}\label{eqQnm}
S_{n,m}:=\prod_{p\mid m}p^{\nu_p(n!)}
\qquad\text{and}\qquad
Q_{n,m}:=\frac{n!}{S_{n,m}}
=\prod_{\substack{p\leq n\\ p\nmid m}}p^{\nu_p(n!)}.
\end{equation}
Thus $Q_{n,m}$ is the largest divisor of $n!$ that is coprime to $m$.

\begin{lem}\label{lemReduction}
Suppose that $1/n!\in K_{m,D}$ and  $Q_{n,m}>1$. Then there exists an integer $r_{n,m}$ with
$1\leq r_{n,m}<Q_{n,m}$ and 
$\gcd(mr_{n,m},Q_{n,m})=1$ such that 
\[
\frac{r_{n,m}}{Q_{n,m}}\in K_{m,D}.
\]
\end{lem}

\begin{proof}
Choose $k\in\N$ sufficiently large that $S_{n,m}\mid m^k$. Since $K_{m,D}$ is $\times m $-invariant (i.e., $mx\pmod{1}\in K_{m,D}$ whenever $x\in K_{m,D}$), we have
\[
\left\{\frac{m^k}{n!}\right\}
=
\left\{\frac{m^k/S_{n,m}}{Q_{n,m}}\right\}
\in K_{m,D}.
\]
Moreover,
\[
\gcd\left(\frac{m^k}{S_{n,m}},Q_{n,m}\right)=1.
\]
Let $r_{n,m}$ be the least positive residue of $m^k/S_{n,m}$ modulo $Q_{n,m}$. Since $Q_{n,m}>1$, we have $1\leq r_{n,m}<Q_{n,m}$ and $\gcd(mr_{n,m},Q_{n,m})=1$. This proves the lemma.
\end{proof}

We also record some standard properties of multiplicative orders.

\begin{lem}\label{LemOrder}
Let $a,q_1,q_2\in\N$ with $\gcd(a,q_1)=\gcd(a,q_2)=1$. Then we have
\begin{enumerate}
\item[(i)] If $q_1\mid q_2$, then $\ord_{q_1}(a)\mid\ord_{q_2}(a)$.
\item[(ii)] If $\gcd(q_1,q_2)=1$, then
$\ord_{q_1q_2}(a)=\lcm\{\ord_{q_1}(a),\ord_{q_2}(a)\}$.
\item[(iii)] $\ord_2(3)=1$, $\ord_{2^2}(3)=2$, and
$\ord_{2^k}(3)=2^{k-2}$ for $k\geq3$. 
\item[(iv)] Let $p_0$ be an odd prime with $p_0\nmid a$. Set $d=\ord_{p_0}(a)$ and  $t=\nu_{p_0}(a^d-1)$.
Then
\[
\ord_{p_0^k}(a)=
\begin{cases}
	d, & 1\leq k\leq t,\\
	d\,p_0^{k-t}, & k>t.
\end{cases}
\]
\end{enumerate}
\end{lem}

The first three assertions are elementary, and assertion (iv) can be found in \cite[Lemma 3]{Bloshchitsyn15R}.

Now we are ready to give the proof of Theorem~\ref{thmmissing}.

\begin{proof}[Proof of Theorem~\ref{thmmissing}]
Fix an odd prime $p_0$ with $p_0\nmid m$, and put $
d=\ord_{p_0}(m)$ and $t=\nu_{p_0}(m^d-1)$.
Choose $n_0>p_0(1+t)$ such that for every $n\geq n_0$,
\begin{equation}\label{eqGeneralBound}
\frac{n}{p_0}-1-t>
\log_{p_0}(n-1)+\log_{p_0}(2m).
\end{equation} 
Such an $n_0$ exists and can be found effectively, as the function $f(x):=\frac{x}{p_0}-1-t-\log_{p_0}(x-1)$ tends to $+\infty$ as $x\to+\infty$ and is increasing for $x>p_0(1+t)$.

We assert that $1/n!\not\in K_{m,D}$ for every $n\geq n_0$. 
Suppose, towards a contradiction, that there exists $n\geq n_0$ such that  $1/n!\in K_{m,D}$. Write $Q_n=Q_{n,m}$. Since $p_0\leq n$ and $p_0\nmid m$, we have $Q_n>1$. Lemmas~\ref{lemReduction} and \ref{Cor} give that
\begin{equation}\label{eqGeneralOrderIneq}
\ord_{Q_n}(m)\leq 2m\cdot\ord_{\rad(Q_n)}(m).
\end{equation}
Let $e_p=\nu_p(n!)$. By \eqref{eqQnm} and Lemma~\ref{LemOrder}(ii),
\[
\ord_{Q_n}(m)
=
\lcm\{\ord_{p^{e_p}}(m):p\leq n,\ p\nmid m\},
\]
and
\[
\ord_{\rad(Q_n)}(m)
=
\lcm\{\ord_p(m):p\leq n,\ p\nmid m\}.
\]
Since $e_{p_0}=\nu_{p_0}(n!)>t$, Lemma~\ref{LemOrder}(iv) gives
\begin{equation}\label{eqLowerOrderGeneral}
\nu_{p_0}(\ord_{Q_n}(m))
\geq e_{p_0}-t.
\end{equation}
On the other hand, $p_0\nmid\ord_{p_0}(m)$, while for every prime $p\leq n$ with $p\nmid m$  and $p\neq p_0$ we have $\ord_p(m)\mid p-1$. Hence
\begin{equation}\label{eqUpperRadGeneral}
\nu_{p_0}(\ord_{\rad(Q_n)}(m))
\leq \log_{p_0}(n-1).
\end{equation}
By Legendre's formula,
\[
e_{p_0}=\nu_{p_0}(n!)
\geq \left\lfloor\frac{n}{p_0}\right\rfloor
\geq \frac{n}{p_0}-1.
\]
Combining this with \eqref{eqGeneralBound}--\eqref{eqUpperRadGeneral}, we obtain
\[
\nu_{p_0}\left(
\frac{\ord_{Q_n}(m)}{\ord_{\rad(Q_n)}(m)}
\right)
>
\log_{p_0}(2m).
\]
Since $\rad(Q_n)\mid Q_n$, the quotient above is a positive integer. It is therefore larger than $2m$, contradicting \eqref{eqGeneralOrderIneq}. Thus $1/n!\notin K_{m,D}$ for every $n\geq n_0$.

It remains to check the finitely many integers $1\leq n<n_0$. Since the base-$m$ expansion of a rational number is eventually periodic, this membership test is effective. Hence the intersection in \eqref{eq1n1KmdF} is finite and can be effectively determined.
\end{proof}

We now specialize the preceding argument to the middle-third Cantor set and prove Theorem~\ref{thmCn!}. For $n\in\N$, write
\[
n!=3^{\nu_3(n!)}M_n.
\]
Then $3\nmid M_n$. If $\frac{1}{n!}\in C$, then since $C$ is $\times 3$-invariant, we see that $\frac{1}{M_n}\in C$. Then Lemma~\ref{Cor} yields that 
\[\ord_{M_n}(3)\leq 6\ord_{\rad(M_n)}(3).\]
However, this inequality does not hold for sufficiently large $n$, as shown in  the following lemma. 

\begin{lem}\label{lemordMn}
For every $n\geq21$,
\[
\ord_{M_n}(3)>6\ord_{\rad(M_n)}(3).
\]
\end{lem}

\begin{proof}
Since $\rad(M_n)\mid M_n$, we have $\ord_{\rad(M_n)}(3)\mid \ord_{M_n}(3)$ by Lemma~\ref{LemOrder}(i). Hence to prove the lemma, it suffices to show that for any $n\geq 21$, 
\begin{equation}\label{eqnu2Mn}
	\nu_2(\ord_{M_n}(3))\geq \nu_2(\ord_{\rad(M_n)}(3))+3.
\end{equation}
Notice that 
\[M_n=\prod_{p\leq n, p\neq 3}p^{e_p}, \quad \rad(M_n)=\prod_{p\leq n, p\neq 3}p,\]
where $e_p=\nu_p(n!)$. Hence by Lemma~\ref{LemOrder}(ii), 
\[\ord_{M_n}(3)=\lcm\left\{\ord_{p^{e_p}}(3): p\leq n,p\neq 3\right\},\]
\[\ord_{\rad(M_n)}(3)=\lcm\left\{\ord_{p}(3): p\leq n,p\neq 3\right\}.\]
By this and the fact that $\ord_p(3)\mid (p-1)$ (Euler's theorem), we have 
\begin{align*}
	\nu_2(\ord_{\rad(M_n)}(3))&=\max\{\nu_2(\ord_p(3)): 3<p\leq n\}\\
	&\leq \max\{\nu_2(p-1): 3<p\leq n\}\\
	&\leq \max\{\log_2(p-1): 3<p\leq n\}\\
	&\leq \log_2(n-1).
\end{align*}
On the other hand, since $\ord_{2^k}(3)=2^{k-2}$ for $k\geq 3$ (cf. Lemma~\ref{LemOrder}(iii)), we have 
\[\nu_2(\ord_{M_n}(3))\geq \nu_2(\ord_{2^{e_2}}(3))\geq e_2-2.\]
By Legendre's formula, 
\[e_2=\nu_2(n!)= \sum_{k=1}^{\infty} \left\lfloor \frac{n}{2^k} \right\rfloor\geq \frac{n}{2}-1.\]
Hence \eqref{eqnu2Mn} holds whenever 
\begin{equation}\label{eqinqn213}
	(\frac{n}{2}-1)-2\geq \log_2(n-1)+3.
\end{equation}
It is easy to see this holds for all $n\geq 21$. 
Hence \eqref{eqnu2Mn} is proved and we complete the proof of the lemma.
\end{proof}

\begin{proof}[Proof of Theorem~\ref{thmCn!}]
By Lemma~\ref{lemordMn} and the argument preceding it, we see that for any $n\geq 21$, $\frac{1}{n!}\not\in C$. A straightforward check shows that $\{1\leq n\leq 20: 1/(n!)\in C\}=\{1,5\}$. 
\end{proof}

Following the proof of Theorem~\ref{thmmissing}, we can give an algorithm to determine the intersection $\{1/(n!):n\in\N\}\cap K_{m,D}$ for an  arbitrary missing-digit set $K_{m,D}$; see Algorithm~\ref{alg:cantor}. See Table~\ref{table} for a few examples.

\begin{algorithm}[hbt]
\caption{Computation of $\left\{\frac{1}{n!}:n\in\N\right\}\cap K_{m,D}$}
\renewcommand{\algorithmicrequire}{\textbf{Input:}}
\renewcommand{\algorithmicensure}{\textbf{Output:}}
\begin{algorithmic}
\Require $m\in\N_{\geq3}$ and $D\subset\{0,1,\ldots,m-1\}$ with $1<\#D<m$
\Ensure The intersection $\left\{\frac{1}{n!}:n\in\N\right\}\cap K_{m,D}$
\State Choose an odd prime $p_0$ with $p_0\nmid m$.
\State Compute $d=\ord_{p_0}(m)$ and $t=\nu_{p_0}(m^d-1)$.
\State Find $n_0>p_0(1+t)$ such that
\[
\frac{n_0}{p_0}-1-t>
\log_{p_0}(n_0-1)+\log_{p_0}(2m).
\]
\State For each $1\leq n<n_0$, determine whether at least one base-$m$ expansion of $1/n!$ uses only digits from $D$.
\State Return the corresponding values of $1/n!$.
\end{algorithmic}
\label{alg:cantor}
\end{algorithm}

\begin{table}[htb]
\centering
\caption{Examples}
\begin{tabular}{c|c}
$(m,D)$&$K_{m,D}\cap\left\{\frac{1}{n!}:n\in\N\right\}$\\
\hline
$(3,\{0,1\})$&$\left\{\frac{1}{2!},\frac{1}{3!},\frac{1}{4!},\frac{1}{6!}\right\}$\\
\hline
$(3,\{0,2\})$&$\left\{1,\frac{1}{5!}\right\}$\\
\hline
$(4,\{0,3\})$&$\{1\}$\\
\hline
$(5,\{0,4\})$&$\left\{1,\frac16\right\}$\\
\hline
$(6,\{0,5\})$&$\left\{1,\frac16\right\}$
\end{tabular}
\label{table}
\end{table}

{\noindent \bf Acknowledgements}. The authors would like to thank Derong Kong, Ruofan Li, and Igor Shparlinski for helpful comments, and especially Igor Shparlinski for suggesting Remark~\ref{reman1}. They also  thank the anonymous referees for suggestions that improved the paper.  Y. F. Wu was supported by Natural Science Foundation of China
(Grant No.~12301110).

\end{document}